\documentclass[12pt,twoside,a4paper,reqno]{amsart}

\usepackage{fullpage}
\usepackage{enumerate}

\usepackage[utf8]{inputenc}
\usepackage[T1]{fontenc}
\usepackage[english]{babel}

\usepackage{amsthm}

\theoremstyle{plain}
\newtheorem{theorem}{Theorem}

\newtheorem{proposition}[theorem]{Proposition}

\newtheorem{lemma}{Lemma}[section]

\newtheorem{question}{Question}

\theoremstyle{remark}

\theoremstyle{definition}

\usepackage{amsmath}
\usepackage{amssymb}
\usepackage{dsfont} 
\usepackage[mathscr]{euscript}

\newcommand{\eps}{\varepsilon}
\newcommand{\Z}{\mathds{Z}}

\newcommand{\N}{\mathds{N}}
\newcommand{\PP}{\mathbb{P}}
\renewcommand{\P}[1]{\PP \left [ #1 \right ]}

\newcommand{\cal}{\mathcal}

\usepackage{tabularx}
\usepackage{graphicx}

\usepackage{constants}
\newconstantfamily{small}{
symbol=c}

\newcommand{\defn}{\textsf}
\title{Homogenization via sprinkling} \author{Itai Benjamini}
\address{The Weizmann Institute, Rehovot, Israel}
\email{itai.benjamini@weizmann.ac.il}

\author{Vincent Tassion}
\address{D\'epartement de Math\'ematiques Universit\'e de Gen\`eve,
  Gen\`eve, Switzerland} \email{Vincent.Tassion@unige.ch}

\thanks{Research supported by the Swiss NSF (VT)}

\date{\today} 

\begin{document}
\maketitle

\begin{abstract}
  We show that a superposition of an $\eps$-Bernoulli bond percolation
  and any everywhere percolating subgraph of $\Z^d$, $d\ge 2$, results
  in a connected subgraph, which after a renormalization   
  dominates  supercritical Bernoulli percolation. This result, which confirms
  a conjecture from \cite{benjamini2000effect}, is mainly motivated by
  obtaining finite volume characterizations of uniqueness for general
  percolation processes.
\end{abstract}

\section{Introduction}
\label{sec:introduction}
Consider a deterministic subset $X$ of the edges of the standard
$d$-dimensional lattice $\mathbb Z^d$, $d\ge 2$. Assume that $X$ is
\defn{percolating everywhere}, meaning that every vertex of $\Z^d$ is
in an infinite connected component of the graph $(\Z^d,X)$. Example of
such graphs are foliation by lines or spanning forests. Consider then
the random set of edges $Y=X\cup\omega$, obtained by adding to $X$ the
open edges $\omega$ of a Bernoulli percolation with density $\eps>0$.
We prove that for every choice of $X$ and every $\eps>0$, the graph
$Y$ is almost surely connected and has large scale geometry similar to
that of supercritical Bernoulli percolation. In
\cite{benjamini2000effect}, this result was already proved in
dimension $d=2$, and conjectured for higher dimensions. The proof of
\cite{benjamini2000effect} for $d=2$ relies strongly on planar
duality, and cannot extend to higher dimensions. In this paper, we
develop new robust methods that allow us to extend the result of
\cite{benjamini2000effect} to any dimension $d\ge 2$. This is the
content of Theorem~\ref{thm:main} below. The main step in our proof is
of independent interest (see Lemma~\ref{lem:stretchExponential}). We
obtain a finite-volume characterization for the uniqueness of the
infinite connected component in~$Y$. More precisely, we show that with
high probability, all the points in the ball of radius $n$ are
connected by a path of $Y$ which lies inside the ball of radius $2n$.
This finite-size criterion approach to uniqueness is in the same
spirit of the original proof of uniqueness for Bernoulli percolation
(see \cite{AKN87} and the recent work of \cite{cerf2013lower}).

As a consequence of the finite-size criterion mentioned above, we show
that a renormalized version of $Y$ dominates highly supercritical
percolation. In particular, $Y$ percolates in sufficiently thick slabs
and in half-spaces. This result is analogous to the Grimmett-Marstrand
theorem~\cite{grimmett1990supercritical}, and we expect most of the
properties of supercritical percolation to hold for $Y$ (see
\cite{grimmett1999}, Chapters 7 and 8). The Grimmett-Marstrand theorem
is a fundamental and powerful tool in supercritical Bernoulli
percolation, but its proof does not provide directly quantitative
estimates 
and relies on the specific symmetries of $\Z^d$ (see
\cite{martineau2013locality} for an extension to some graphs with less
symmetries than the canonical~$\Z^d$). We hope that the method in the
present paper could be useful to obtain quantitative and robust proofs
of the Grimmett-Marstrand theorem.

We do not require any symmetry hypothesis on the set $X$. Thus, the
percolation process $Y$ is not necessarily invariant under the
symmetries of $\Z^d$. Therefore, the uniqueness of the infinite
cluster cannot be derived from the Burton-Keane
theorem~\cite{burton1989uniqueness}. In this sense, our result can be
seen as a generalization of the Burton-Keane theorem. Recently,
Teixeira \cite{teixeira2014percolation} considered general percolation
processes with high marginals on graphs with polynomial volume growth.
Under some additional assumptions but without requiring symmetries or
invariance, he obtains uniqueness of the infinite cluster. Since $Y$
does not necessarily have high marginals, our uniqueness result is not
implied by Teixeira's work.

We also obtain that the critical value for Bernoulli percolation on
$Y$ satisfies $p_c(Y)<1$. This is related to the initial motivation
in~\cite{benjamini2000effect} for introducing the random set $Y$.
Proving that $p_c(Y)<1$ can be seen as an intermediate question in
order to understand under which conditions a random infinite subgraph
$G$ of $\Z^d$ has $p_c(G)<1$. Finding such conditions is very
challenging, and related to famous open problems in percolation,
e.g.\@ the absence of infinite cluster at criticality for Bernoulli
percolation (see \cite{benjamini2000effect} for more details). We give
related questions in Section~\ref{sec:questions}.

\subsection{Main results}
\label{sec:main-results}

We prove the following theorem, which confirms a conjecture of
Benjamini, Häggström and Schramm \cite{benjamini2000effect}. (If
needed, see Section~\ref{sec:notation-definitions} for notation and
definitions.)

\begin{theorem}
\label{thm:main}
  Let $X$ be a fixed everywhere percolating subgraph of $\Z^d$,
and let $Y = X\cup \omega$ be obtained from $X$ by adding an
$\eps$-percolation $\omega$.
For any $\eps > 0$, the following hold.
\begin{enumerate}[(i)]
\item \label{item:1} The subgraph $Y$ is connected a.s.
\item \label{item:2} The critical parameter for Bernoulli percolation
  on $Y$ satisfies $p_c(Y)<1$ a.s. 
\item \label{item:3} The subgraph $Y$ percolates in the upper
half-space a.s.
\item \label{item:4} There exists $L=L(\eps,d)$ such that $Y$
  percolates in the slab $\Z^2\times \{0,\ldots,L\}$ a.s.
\item\label{item:5} For every fixed $p<1$, a renormalized version of
  $Y$ stochastically dominates a $p$-Bernoulli percolation.
\end{enumerate}
\end{theorem}

The precise signification of Item~(\ref{item:5}) is the following.
Define the percolation process $Y^{(n)}$ on $\Z^d$ by declaring an
edge $e=\{x,y\}\in \mathds E^d$ open if the vertex $2nx$ is
$Y$-connected to $2ny$ inside $n(x+y)+\Lambda_{2n}$.
Item~(\ref{item:5}) occurs if for every $p<1$, there exists $n\ge 1$
such that the process $Y^{(n)}$ dominates stochastically a
$p$-Bernoulli percolation (see \cite{ligget1997domination} for more
details on stochastic domination).

\noindent\textbf{Remarks.}

\textbf{a.} As in \cite{benjamini2000effect}, a straightforward extension of
our proof shows that an analogue of Theorem~\ref{thm:main} holds if we
only assume that $X$ is densely percolating. A subgraph $X$ is said to
be densely percolating if there exists $R$ such that every box
$2R.z+\Lambda_R$, $z\in \Z^2$, intersects an infinite component of
$X$. In this framework, Item~(\ref{item:1}) needs to be replaced by
uniqueness of the infinite cluster, and the definition of
renormalization in Item~(\ref{item:5}) has to be slightly adapted. 

\textbf{b.} Once we know that a rescaled version of $Y$ dominates
supercritical Bernoulli percolation, one could use the known results
for Bernoulli percolation to obtain other properties of $Y$ (see
\cite{grimmett1999}, Chapters 7 and 8).


\subsection{Questions}
\label{sec:questions}

Let us begin by rewriting the question of \cite{benjamini2000effect}
that motivates the problem studied here.

\begin{question}
  \label{ques:1}
  Is there an invariant, finite energy percolation $X$ on $\Z^d$, which
a.s.\@ percolates and satisfies $p_c(X) = 1$?
\end{question}
An invariant percolation is a probability measure on the percolation
configuration that is invariant under the symmetries of $\Z^d$. The
finite energy property was considered in~\cite{newman1981infinite}. In
the sense of Lyons and Schramm~\cite{lyons1999indistinguishability},
it corresponds to insertion and deletion tolerance: given an edge $e$,
the conditional probability that $e$ is present (resp.\@ absent) given
the status of all the other edges is positive. Benjamini, Häggström
and Schramm~\cite{benjamini2000effect} showed that
Question~\ref{ques:1} has a positive answer if we replace the finite
energy condition by the insertion tolerance. They construct an
insertion tolerant invariant process $X$ (obtained by adding and
$\eps$-percolation to a well-chosen invariant percolation), that
percolates but satisfies $p_c(X)=1$.

Adding an $\eps$-percolation to a percolation process is an easy way
to build insertion tolerant processes. Of course, one may add a more
general process instead. Our proof of Theorem~\ref{thm:main} uses
strongly that our process was constructed by adding an
$\eps$-percolation. With Question~\ref{ques:1} in mind, it would be
interesting to understand the effect of adding a more general process.
We suggest the following question.
\begin{question}
  \label{ques:2}
  Let $X$ be a fixed everywhere percolating subgraph of $\Z^d$. Let
  $\eta$ be a percolation process such that $\eta\neq \emptyset$
  almost surely. Assume that $\eta$ is ergodic with respect to the
  translations and invariant with respect to the whole automorphisms
  of the grid. Which properties among (\ref{item:1}), (\ref{item:2}),
  (\ref{item:3}), (\ref{item:4}) and (\ref{item:5}) are satisfied by
  $Y:=X\cup\eta$?
\end{question}
Let us give a particular case. In $\Z^3$, consider a superposition of
two independent ergodic invariant spanning forests. It is a.s.\@
connected?

Another natural generalization of the problem treated in this paper is
to consider graphs other than the hypercubic lattice. In
\cite{benjamini2001uniform} it is shown that non-amenable graphs,
admit a spanning forest that stays disconnected after adding the edges
of an $\eps$-percolation, for $\eps$ sufficiently small. A
positive answer to the following question would show that, in the
context of transitive graphs with $p_c<1$, the existence of an
everywhere percolating disconnected subgraph that remains disconnected
after adding an $\eps$-sprinkling is equivalent to non-amenability.
\begin{question}
  \label{ques:3}
  Let $G$ be a transitive amenable graph with critical value for
  Bernoulli percolation satisfying $p_c(G)<1$. Let $X$ be an
  everywhere percolating subgraph of $G$, and let $Y = X\cup \omega$
  be obtained from $X$ by adding an $\eps$-percolation $\omega$. Is
  $Y$ connected almost surely?
\end{question}
As an intermediate step toward Question 3, one can start first with
transitive graphs of polynomial volume growth (a framework in which
our methods are more likely to be adapted).

Another perspective is to study the simple random walk on a
superposition of an everywhere percolating subgraph of the lattice and
an independent sprinkling (for example, verify diffusivity in this
framework). Adding the sprinkling can be viewed as homogenisation,
this suggests to study spaces of harmonic functions on this
environment (in the spirit of~\cite{benjamini2011disorder}).
\subsection{Organization of the paper}
\label{sec:organization-paper}

For the rest of the paper, \emph{we fix the values of $d\ge 2$ and
  $\eps>0$}. 

\noindent\textbf{Constants.} In the proof, we will introduce \defn{constants}, denoted
by $C_0,C_1,\ldots\ $ By convention, the constants are elements of
$(0,\infty)$, they may depend on $d$ and $\eps$, but never depend on
any other parameter of the model. In particular, they never depend on
the chosen everywhere percolating subgraph $X$, or the size of the box
$n$.

In Section~\ref{sec:perc-highly-conn}, we study the effect of an
$\eps$-percolation on a finite highly connected graph. More precisely
we consider a graph $G$ with $O(N^{2d})$ vertices that is
$N$-connected ($G$ is connected and remains connected if we erase any
set of $N$ edges). We show that an $\eps$-percolation on such a graph
is connected with high probability.

The main new ideas are presented in Section~\ref{sec:proof-theorem},
where we prove the following lemma. 
\begin{lemma}
  \label{lem:stretchExponential}
  Let $X$ be a fixed everywhere percolating subset of edges of $\Z^d$.
  Let $Y = X\cup \omega$ be obtained from $X$ by adding an
  $\eps$-percolation $\omega$. For every $n\ge 1$, we have
 \begin{equation}
   \label{eq:1}
   \P{\text{For all $x,y\in \Lambda_n$, $x$ is $Y$-connected to $y$
       inside $\Lambda_{2n}$}}\ge 1-C_3e^{-C_1 \sqrt n}  
 \end{equation}
\end{lemma}
Let us give the strategy used to prove
Lemma~\ref{lem:stretchExponential}. The restriction of an everywhere
percolating subgraph of $\Z^d$ to the finite box $\Lambda_n$ gives a
partition of $\Lambda_n$ into finite clusters. If we forget about
possible small clusters at the boundary of $\Lambda_n$ and assume that
all the finite clusters partitioning $\Lambda_n$ intersect the box
$\Lambda_{n-\sqrt n}$, then contracting these clusters result in a
highly connected graph. Indeed, a non-trivial union of clusters must
``cross'' the annulus $\Lambda_n\setminus \Lambda_{n-\sqrt{n}}$, which
implies that its boundary must have size at least $\sqrt n$. We can
thus apply the result of Section~\ref{sec:perc-highly-conn}. The main
difficulty is to treat carefully the small clusters at the boundary of
$\Lambda_n$.

In Section~\ref{sec:theor-impl-theor}, we deduce
Theorem~\ref{thm:main} from Lemma~\ref{lem:stretchExponential}. That
section uses standard renormalization and stochastic domination tools,
which were already used by Benjamini, Häggström and Schramm
\cite{benjamini2000effect} to treat the case $d=2$.

\subsection{Notation, definitions}
\label{sec:notation-definitions}

Let $(\Z^d,\mathds E^d)$, $d\ge2$, be the standard $d$-dimensional hypercubic
lattice. For $r,R>0$ (not necessarily integer), we write
$\Lambda_r:=[-r,r]^d\cap \Z^d$ and $A_{r,R}:=\Lambda_{R}\setminus
\Lambda_{r}$. For $z\in \Z^d$, we denote by
$\Lambda_r(z):=z+\Lambda_r$ the translate of $\Lambda_r$ by vector~$z$.

In this paper, we call \defn{subgraph of $\Z^d$} a set of edges
$X\subset \mathds E^d$, and identify it to the graph $(\Z^d,X)$.
(Notice that the vertex set of every subgraph of $\Z^d$ considered
here will always be the entirety of $\Z^d$.) We say that two vertices
$x,y\in \Z^d$ are \defn{$X$-connected}, if there exists a sequence of
disjoint vertices $x_1,\ldots,x_\ell \in \Z^d$, such that $x_1=x$,
$x_\ell=y$, and for every $1\le i<\ell$ the edge $\{x_i,x_{i+1}\}$
belongs to $X$. We say that $x$ and $y$ are \defn{$X$-connected inside
  $S\subset \Z^d$} if, in addition, all the $x_i$'s belong to $S$.

A subgraph $X$ is said to be \defn{everywhere percolating} if all its
connected components are infinite. In other words, a subgraph is
everywhere percolating if for every vertex $x$ in $\Z^d$, $x$ is
$X$-connected to infinity. We say that $X$ \defn{percolates in
  $S\subset\Z^d$} if the graph induced by $X$ on $S$ contains an
infinite connected component.

We call \defn{$p$-Bernoulli percolation} (or simply
\defn{$p$-percolation}) the random subgraph $\omega$ of $\Z^d$,
constructed as follows: each edge of $\mathds E^d$ is examined
independently of the other and is declared to be an element of
$\omega$ with probability $p$.

\section{Percolation on a highly connected finite graph}
\label{sec:perc-highly-conn}

In this section, we show that an $\eps$-percolation on any
$N$-connected finite graph with $O(N^{2d})$ vertices connects all the
vertices with large probability. Let us begin with the definitions
needed to state the result. A finite \defn{multigraph} is given by a
pair $G=(V,E)$: $V$ is a finite set of vertices, and $E$ is a multiset
of pairs of unordered vertices (multiple edges and loops are allowed).
We work with multigraphs rather than standard graphs because we will
consider multigraphs obtained from other graphs by \emph{contraction},
and we want to keep track of the multiple edges created by the
contraction procedure. A multigraph $G=(V,E)$ is said to be
\defn{$N$-connected}\footnote{In graph theory, the terminology
  ``$N$-edge-connected'' is used in this case to distinguish with
  vertex-connectivity.} if it is connected, and removing any set of
$N$ edges does not disconnect it. In other words, for any subset
$F\subset E$ such that $|F|\le N$, the graph $(V,E\setminus F)$ is
connected. An $\eps$-percolation on $G$ is defined equivalently as on
$\Z^d$: it is a random set of edges $\omega\subset E$ such that the
events $\{e\in \omega\}$ are independent, each of them having
probability equal to $\eps$.

\begin{proposition}
   \label{prop:combinatoric}
   Let $N\in\N$. Let $G=(V,E)$ be an $N$-connected multigraph with
   \mbox{$|V|\le N^{2d}$}. Let $\omega \subset E$ be an
   $\eps$-percolation on $G$. Then
   \begin{equation*}
      \P{\text{The graph $(V,\omega)$ is connected}}\ge1- C_2e^{-C_1 N}.
   \end{equation*}
\end{proposition}

We begin with the following lemma, which says that adding an
$(\eps/4d)$-percolation on a subgraph of $G$ either connects all the
vertices or shrinks the number of connected components by a factor
smaller than $1/\sqrt{N}$ (with large probability). We will then prove
the proposition by applying this lemma $4d$ times. Given $X\subset E$,
we write $K(X)$ for the number of connected components in the graph
$(V,X)$.

\begin{lemma}
  \label{lem:shrink}
   Let $\omega_1$ be an $(\eps/4d)$-percolation on
  $G$. For every $X \subset E$, we have 
  \begin{equation}
    \label{eq:2}
    \P{K(X\cup \omega_1)> 1\vee \frac{ K(X)}{\sqrt N} }\le C_0 e^{-C_1N},
  \end{equation}
\end{lemma}

\begin{proof}
  We can assume that $K(X) >1$ (if $K(X)=1$, then
  Equation~\eqref{eq:2} is trivially true). We say that a set of
  vertices $S\subset V$ \defn{generated by $X$}, if it can be exactly
  written as a disjoint union
  \begin{equation}
    S= S_1 \cup\cdots\cup S_m,\label{eq:3}
  \end{equation}
  where $S_1,\ldots,S_m$ are disjoint connected components of $X$. For
  such a set $S$, we define $\mathsf m (S):=m$ as the number of
  connected components in the decomposition \eqref{eq:3}. Notice that
  every non-empty set generated by $X$ satisfies $1\le \mathsf m
  (S)\le K(X)$. The case $\mathsf m (S)=1$ corresponds to $S$ being a
  single connected component of $(V,X)$, and $\mathsf m (S)=K(X)$
  corresponds to $S=V$.
  
  If $K(X\cup \omega_1)> 1\vee K(X)/\sqrt N$, then there must exist a
  connected component $S\subsetneq V$ of $X\cup \omega_1$ that
  satisfies $1\le \mathsf m(S) <\sqrt N$. By the union bound, we
  obtain
  \begin{equation}
    \label{eq:4}
    \P{K(X\cup \omega_1)> 1\vee \frac{ K(X)}{\sqrt N} }\le
   \!\!\! \sum_{\substack{S\subsetneq V\\
        1\le \mathsf m(S)<\sqrt N}} \!\!\!\P{\text{$S$ is a connected component of $X \cup \omega_1$}}.
  \end{equation}
  Now, if a non-empty set $S\subsetneq V$ is a connected component of $X
  \cup \omega_1$, then all the edges connecting a vertex of $S$ to a
  vertex in $V\setminus S$ must be $\omega_1$-closed. Since the graph
  $G$ is $N$-connected there are at least $N$ such edges, and we obtain
  \begin{equation}
    \label{eq:5}
    \P{\text{$S$ is a connected component of $X \cup
        \omega_1$}}\le (1-\eps/4d)^N.
  \end{equation}
  Plugging Equation~\eqref{eq:5} in \eqref{eq:4}, we find
  \begin{align*}
    \P{K(X\cup \omega_1)> 1\vee \frac{ K(X)}{\sqrt N} }&\le
    |\{S\::\:  1\le \mathsf m(S)<\sqrt N\}|\:  (1-\eps/4d)^N\\
    &\big. \le  \sum_{1\le k< \sqrt N} \binom {K(X)} k \:  (1-\eps/4d)^N\\
    &\bigg. \le  \sqrt N  N^{2d\sqrt N} (1-\eps/4d)^N\\
    &\bigg.\le C_0 e^{-C_1N}.
  \end{align*}
  In the third line we use that $K(X)\le |V| \le N^{2d}$ to bound the
  binomial coefficient by $N^{2d\sqrt N}$.
\end{proof}

\begin{proof}[Proof of Proposition~\ref{prop:combinatoric}]
  Let $\omega_1,\ldots,\omega_{4d}$ be $4d$ independent
  $(\eps/4d)$-percolations on $G$, set $\eta_0=\emptyset$ and
  $\eta_k:=\omega_1\cup\ldots\cup\omega_k$. By Lemma~\ref{lem:shrink} we have,
  for all $1\le k\le 4d$,
  \begin{equation*}
    \P{K(\eta_k)> 1\vee \big(K(\eta_{k-1})/\sqrt N\big) }\le C_0 e^{-C_1N}.
  \end{equation*}
  Since $K(\eta_0)=|V|\le N^{2d}$, we find by induction 
  \begin{equation*}
         \P{K(\eta_{k})> N^{2d-k/2}}\le  kC_0 e^{-C_1N}.
  \end{equation*}
  Setting $k=4d$ in the equation above, we obtain that $\eta_{4d}$ is
  connected with probability larger than $1-C_2 e^{-C_1N}$. This
  concludes the proof, since $\eta_{4d}$ is stochastically dominated
  by an $\eps$-percolation.
\end{proof}

\section{Proof of Lemma~\ref{lem:stretchExponential}}
\label{sec:proof-theorem}
Let $X$ be an everywhere percolating subgraph of $\Z^d$, let $n\ge 1$.
We define $\cal C_{8dn}(X)$ as the set of connected components for the
graph induced by $X$ on the box $\Lambda_{8dn}$ (two vertices are in
the same connected component if they are $X$-connected inside
$\Lambda_{8dn}$). Fix $n_0$ such that, for every $n\ge n_0$, the size
the boundary of $\Lambda_{8dn}$ is smaller than $n^d$. (We call
boundary of $\Lambda_n$ the set $\Lambda_{n}\setminus\Lambda_{n-1}$).
Notice that $n_0$ depends only on the dimension $d$. Since any element
of $\cal C_{8dn}(X)$ contains at least a vertex at the boundary of
$\Lambda_{8dn}$, we also have, for every $n\ge n_0$,
\begin{equation}
|\cal C_{8dn}(X)|\le n^d.\label{eq:6}
\end{equation}
For $0< a\le b\le 8dn$, define $U_{a,b}(X)$ as the number of sets
$C\in \mathcal C_{8dn}(X)$ such that $C\cap \Lambda_a=\emptyset$ and
$C\cap \Lambda_b\neq\emptyset$. In other words,
\begin{equation*}
   U_{a,b}(X)=\mathrm{Card}\{C\in \Lambda_{8dn} \::\: 
  a< \mathrm d(0,C)\le b\},
\end{equation*}
where $\mathrm d(0,C)$ denotes the $L^\infty$-distance between the
origin and the set $C$. Define $U_{0,b}$ as the number of sets $C\in
\mathcal C_{8dn}(X)$ such that $C\cap\Lambda_b\neq\emptyset$. That is
\begin{equation*}
  U_{0,b}(X)=\mathrm{Card}\{C\in \Lambda_{8dn} \::\: \mathrm d(0,C)\le b\}.
\end{equation*}
Notice that $U_{a,b}(X)$ depends on $n$, but we keep this dependence
implicit to lighten the notation.

\begin{lemma}
  \label{lem:annulusToBulk}
  Fix an everywhere percolating subgraph $X$ of $\Z^d$. Let $n_0\le
  n\le m\le m+\sqrt n\le 8dn$. Let $\omega$ be an $\eps$-percolation
  restricted to $A_{m,m+2\sqrt n}$, then
  \begin{equation*}
     \P{U_{0,m}(X\cup \omega) > 1\vee U_{m,m+2\sqrt n}(X)} \le C_2
    e^{-C_1\sqrt n}.   
  \end{equation*}
\end{lemma}
\noindent(In the statement of the lemma, the \defn{$\eps$-percolation restricted
  to $A_{m,m+2\sqrt n}$} is defined by $\omega=\eta\cap A_{m,m+2\sqrt
  n}$, where $\eta$ is an $\eps$-percolation in $\Z^d$.)

\begin{proof}
  Let us first assume $U_{m,m+\sqrt n}(X)=0$. In this case, the proof
  is easier and we directly show that
  \begin{equation}
    \label{eq:7}
    \P{U_{0,m}(X\vee\omega')>1}\le  C_2e^{-C_1\sqrt n}, 
  \end{equation}
  where $\omega'=\omega \cap A_{m,m+\sqrt n}$ is the restriction of
  $\omega$ to the annulus $A_{m,m+\sqrt n}$. 

  \noindent Let $G$ be the graph obtained from $\Lambda_{m+\sqrt n}$ by the
  following contraction procedure.
  \begin{enumerate}
  \item Start with $\Lambda_{m+\sqrt n}$, with the standard graph structure
    induced by $\Z^d$.
  \item Examine the elements of $\cal C_{8dn}(X)$ one after the other. For
    every $C\in \cal C_{8dn}(X)$, contract all the
    points $x\in C\cap \Lambda_{m+\sqrt n}$ into one vertex.
  \end{enumerate}
  Since $U_{m,m+\sqrt n}(X)=0$, the graph $G$ has exactly $U_{0,m}(X)$
  vertices, and $U_{0,m}(X\vee\omega')$ corresponds the number of
  connected components resulting from an $\eps$-percolation on $G$.
  Therefore, Equation~\eqref{eq:7} follows from
  Proposition~\ref{prop:combinatoric}, applied to $G$ with $N=\sqrt
  n$. Indeed, the graph $G$ has at most $(\sqrt n)^{2d}$ vertices by
  \eqref{eq:6}, and is $\sqrt n$-connected. To see that $G$ is $\sqrt
  n$-connected, observe that any non-trivial union of elements of
  $\mathcal C_{8dn}(X)$ that intersect $\Lambda_{m+\sqrt n}$ must ``cross''
  the annulus $A_{m,m+\sqrt n}$ (this follows from the hypothesis $U_{m,m+\sqrt n}(X)=0$).
  
  We now turn to the case $U_{m,m+\sqrt n}(X)> 0$, in  which we show
  \begin{equation}
    \label{eq:8}
    \P{U_{0,m}(X\cup \omega)\le U_{m,m+2\sqrt n}(X)} \ge 1-C_2e^{C_1\sqrt n}.
  \end{equation}
  The strategy is very similar to the one used in the case
  $U_{m,m+\sqrt n}(X)=0$, except that here we need to consider
  carefully the clusters that intersect the annulus $A_{m,m+2\sqrt n}$
  but not the box $\Lambda_m$. More precisely, we partition the
  elements of $\cal C_{8dn}(X)$ \emph{intersecting the box
    $\Lambda_{m+2\sqrt n}$}, into the following two types:
  \begin{itemize}
  \item the \defn{bulk-clusters}, defined as the elements  $C\in\cal
    C_{8dn}(X)$ such that $C\cap \Lambda_m\neq \emptyset$;
  \item the \defn{boundary-clusters}, defined as the elements  $C\in\cal
    C_{8dn}(X)$ such that $C\cap \Lambda_m=\emptyset$.
  \end{itemize}
  Notice that
  \begin{equation*}
        U_{0,m+2\sqrt n}(X)=U_{0,m}(X)+U_{m,m+2\sqrt n}(X),
  \end{equation*}
  where $U_{0,m}(X)$ counts the bulk-clusters, and $U_{m,m+2\sqrt
    n}(X)$ counts the boundary-clusters.

  Define $\tilde C$ as the union of all the boundary-clusters. The
  hypothesis $U_{m,m+\sqrt n}(X)> 0$ implies that at least one
  boundary-cluster intersects the annulus $A_{m,m+\sqrt n}$. Thus, we
  have
  \begin{equation*}
       \tilde C \cap \Lambda_{m+\sqrt n} \neq \emptyset.
  \end{equation*}
  We now construct a $\sqrt n$-connected  graph $G$ by the following
  contraction procedure.
   \begin{enumerate}
  \item Start with $\Lambda_{m+2 \sqrt n}$, with the graph structure
    induced by $\Z^d$.
  \item For every bulk-cluster $C\in \cal C_{8dn}(X)$, contract all
    the points $x\in C\cap \Lambda_{m+2\sqrt n}$ into one vertex.
  \item Contract all the vertices $x$ that belong
    to $\tilde C \cap \Lambda_{m+2\sqrt n}$  into one vertex. 
  \end{enumerate}
  To see that the graph is $\sqrt n$-connected, observe that all the
  bulk-clusters and $\tilde C$ cross the annulus $A_{m+\sqrt n,
    m+2\sqrt n}$. Proposition~\ref{prop:combinatoric} applied with
  $N=\sqrt n$, ensures than an $\eps$-percolation on $G$ connects all
  its vertices with probability larger than $1-C_2e^{-C_1\sqrt n}$. This
 proves Equation~\eqref{eq:8}, because the $\eps$-percolation $\omega$ can be
  interpreted as an $\eps$-percolation on $G$, except that the
  $U_{m,m+2\sqrt n}(X)$ boundary-clusters were ``artificially'' merged
  into one point in the construction of $G$.
\end{proof}

\begin{lemma}
\label{lem:annulusShrink}
Fix an everywhere percolating subgraph $X$ of $\Z^d$. Let $n_0\le n\le
m\le m+ 2n\le 8dn$. Let $\omega $ be an $\eps$-percolation restricted
to the annulus $A_{m,m+2n}$, then
  \begin{equation}
    \label{eq:9}
    \P{U_{0,m}(X\cup \omega) > 1\vee \frac{U_{0,m+2n}(X)}{\sqrt n}} \le C_2
    e^{-C_1\sqrt n}.   
  \end{equation}
\end{lemma}
\begin{proof}
  First observe that
  \begin{equation*}
     U_{0,m}(X)+\sum_{i=0}^{\lfloor \sqrt n\rfloor-1}U_{m+2i\sqrt n,
      m+2(i+1)\sqrt n}(X)\le U_{0,m+2n}(X).
  \end{equation*}
  Among the $1+\lfloor \sqrt n\rfloor$ terms summed on the left hand
  side, at least one of them must be smaller than or equal to
  $U_{m,m+2n}(X)/\sqrt n$. If $U_{0,m}(X)\le U_{m,m+2n}(X)/\sqrt n $,
  then Equation~\eqref{eq:9} is trivially true. Otherwise, one can
  fix $i$ such that $U_{m+2i\sqrt n, m+2(i+1)\sqrt n}(X)\le
  U_{m,m+2n}(X)/\sqrt n$. By Lemma~\ref{lem:annulusToBulk}, we have
  \begin{equation*}
   {U_{0,m+2i\sqrt n}(X\cup \omega) >1\vee\frac  {U_{0,m+2n}(X)}{\sqrt
        n}}\le C_2e^{-C_1\sqrt n}.
  \end{equation*}
  Then, use the inequality $U_{0,m}(X\cup \omega)\le U_{0,m+2i\sqrt n
  }(X\cup \omega)$ to conclude the proof.
\end{proof}

\begin{proof}[Proof of Lemma~\ref{lem:stretchExponential}]
Since $n_0$ depends only on $d$, it is sufficient to prove
Equation~\eqref{eq:1} in
Lemma~\ref{lem:stretchExponential} for $n\ge n_0$ (recall that $n_0$
was defined at the beginning of Section~\ref{sec:proof-theorem}). We
wish to apply Lemma~\ref{lem:annulusShrink} recursively in the $2d$
disjoint annuli
\[A_{(8d-2)n,8dn},A_{(8d-4)n,(8d-2)n},\ldots,A_{4dn,(4d+2)n}.\] For
$i=1,\ldots, 2d$, set $m(i)=(8d-2i)n$, and $\omega_i=\omega \cap
\Lambda_{m(i),8dn}$. By Lemma~\ref{lem:annulusShrink}, we have for all
$i<2d$
\begin{equation}
  \label{eq:10}
  \P{U_{0,m(i+1)}(X\cup \omega_{i+1})> 1\vee \frac{U_{0,m(i)}(X\cup \omega_i)}{\sqrt n}}\le
  C_2e^{-C_1\sqrt n}.
\end{equation}
By Equation~\eqref{eq:6}, we have $U_{0,8dn}(X)\le n^d$ for all $n\ge
n_0$. This implies that
\begin{align*}
  \P{U_{0,4dn}(X\cup \omega)=1}
  &\ge \P{\text{For all $i$, } U_{0,m(i+1)}(X\cup\omega_{i+1})\le 1\vee
    \frac{U_{0,m(i)}(X\cup\omega_i)}{\sqrt n}}\\
  &\ge 1-2d C_2e^{-C_1\sqrt n},
\end{align*}
where the last line follows from Equation~\eqref{eq:10}. This proves
that with probability larger than $1-C_3e^{-C_1\sqrt n}$, all the
vertices of $\Lambda_{4n}$ are $(X\cup\omega)$-connected inside
$\Lambda_{8n}$. Lemma~\ref{lem:stretchExponential} follows
straightforwardly.
\end{proof}

\section{Proof of  Theorem~\ref{thm:main}}
\label{sec:theor-impl-theor}

  Let $X$ be a fixed everywhere percolating subgraph of $\Z^d$ ,
and let $Y = X\cup \omega$ be obtained from $X$ by adding an
$\eps$-percolation $\omega$. In this section, we will show how to
derive Theorem~\ref{thm:main} from the following estimate, stated in Lemma~\ref{lem:stretchExponential}.
 \begin{equation*}
    \P{\text{For all $x,y\in \Lambda_n$, $x$ is $Y$-connected to $y$
       inside $\Lambda_{2n}$}}\ge 1-C_3e^{-C_1 \sqrt n}  
 \end{equation*}
 Recall that the constants $C_1,C_3$ do not depend on the underlying
 everywhere percolating graph $X$. Thus, by considering translates of
 $X$, we show that for every $z\in \Z^d$,
\begin{equation}
  \label{eq:11}
  \P{\text{For all $x,y\in \Lambda_n(z)$, $x$ is $Y$-connected to $y$
       inside $\Lambda_{2n}(z)$}}\ge 1-C_3e^{-C_1 \sqrt n}  
\end{equation}

\subsection{Proof of Item~(\ref{item:5})}
\label{sec:proof-item-5}
Let $p<1$. By Corollary 1.4 in \cite{ligget1997domination}, one can
pick $p'<1$ such that any $3$-dependent\footnote{A percolation process $Z$
  is said to be $3$-dependent if, given  two sets of edges $A$ and $B$
  such that any edge of $A$ is at $L^\infty$ distance at least $3$
  from any edge of $B$, the processes $Z\cap A$ and $Z\cap B$
  are independent.} percolation process $Z$ on
$\Z^d$, satisfying for every edge $e\in \mathds E^d$
\[\P{e\in Z}>p'\]
dominates stochastically a $p$-Bernoulli percolation.

Recall that the process $Y^{(n)}$ is defined by setting $\{x,y\}\in
Y^{(n)}$ if $2nx$ is $Y$-connected to $2ny$ inside
$\Lambda_{4n}(nx+ny)$. One can easily verify that for every $n\ge1$,
the process $Y^{(n)}$ is $3$-dependent. Thus, in order to show that for some $n$,
$Y^{(n)}$ dominates a $p$-Bernoulli percolation, one only need to
prove that for every edge $\{x,y\}\in \mathds E^d$, 
\begin{equation}
  \label{eq:12}
  \P{\text{$2nx$ is $Y$-connected to $2ny$ inside $\Lambda_{2n}(nx+ny)$}}>p'.
\end{equation}
Since both $2nx$ and $2ny$ belong to $\Lambda_{n}(nx+ny)$, we get that
Equation~\eqref{eq:12} holds for $n$ large enough, by applying
Equation~\eqref{eq:11} with $z=nx+ny$.

\subsection{Proof of Items~(\ref{item:3}) and (\ref{item:4})}
\label{sec:proof-item-3-and-4}

Let $p>p_c(\Z^2)$. By Item~(\ref{item:5}), one can pick $n$ such that
$Y^{(n)}$ dominates a $p$-Bernoulli percolation on $\Z^d$. It is known
that a $p$-Bernoulli percolation percolates in the half-plane
$\{2,3,\ldots\}\times\Z\times\{0_{\Z^{d-2}}\}$ (see e.g.~\cite{kesten1982percolation}). By stochastic
domination, the same holds for $Y^{(n)}$. This implies that $Y$
percolates in the half space $\N\times\Z^{d-1}$. When $d\ge 3$, it
also implies that  $Y$
percolates in the slab $\Z^2\times \{-2n,\ldots,2n\}^{d-2}$.

\subsection{Proof of Item~(\ref{item:2})}
\label{sec:proof-item-2}

We begin as in the proof of Item~(\ref{item:5}). By stochastic
domination arguments, one can fix $p'<1$ such that any $3$-dependent
percolation process $Z$ with $\P{e\in Z}>p'$ percolates in $\Z^d$. Let
$p''\in(p',1)$. By
Equation~\eqref{eq:11}, one can fix $n$ such that for every edge
$\{x,y\}\in \mathds E^d$, 
\[
 \P{\text{$2nx$ is $Y$-connected to $2ny$ inside $n(x+y)+\Lambda_{2n}$}}>p''.
\]
Let $\eta_q$ be a $q$-Bernoulli percolation process in $\Z^d$,
independent of $\omega$, and set $Y_q=Y\cap \eta_q$. This way, $Y_q$
corresponds exactly to a $q$-percolation on $Y$. Choose $q<1$ such
that all the edges of $\Lambda_{2n}$ belong to $\eta_q$ with
probability larger than $1-(p''-p')$. This way, for fixed $z$, the
processes $Y_q$ and $Y$ differs in the box $z+\Lambda_{2n}$ with
probability smaller than $(p''-p')$. Thus, for every edge
$\{x,y\}\in\mathds E^d$, we have
\begin{align*}
    &\P{\text{$2nx$ is $Y_q$-connected to $2ny$ inside
      $n(x+y)+\Lambda_{2n}$}}\\
&\quad \ge  \P{\text{$2nx$ is $Y$-connected to $2ny$ inside
      $n(x+y)+\Lambda_{2n}$}} -(p''-p')\\
&\quad >p'
\end{align*}
By the same stochastic argument that we already used several times,
this is enough to guarantee that the process $Y_q$ percolates in
$\Z^d$. This proves that $p_c(Y)\le q<1$ almost surely.

\subsection{Proof of Item~(\ref{item:1})}
\label{sec:proof-item-1}
Let $\cal U_n$ be the event that all the pairs of points in $\Lambda_n$
are $Y$-connected in $\Z^d$. By
Lemma~\ref{lem:stretchExponential}, we have
\begin{equation*}
    \sum \P{\cal U_n^c}<\infty.
\end{equation*}
Thus, by the Borel-Cantelli Lemma, we have $\P{\cup_{n}\cap_{m\ge n} \cal
  U_{m}}=1$, which implies that $Y$ is almost surely connected.

\section{Acknowledgements}
\label{sec:acknoledgement}
We are grateful to Matan Harel for his careful reading of a
preliminary version of this paper.

\bibliographystyle{alphaNames}
\bibliography{references}

\end{document}